\documentclass{amsart}

\usepackage{bbding}
\usepackage{amsmath}
\usepackage{amssymb,amsfonts,amsthm}
\usepackage[all]{xy}

\newtheorem{theorem}{Theorem}[section]
\newtheorem{lemma}[theorem]{Lemma}
\newtheorem{proposition}[theorem]{Proposition}
\newtheorem{corollary}[theorem]{Corollary}
\theoremstyle{definition}
\newtheorem{definition}[theorem]{Definition}
\newtheorem{example}[theorem]{Example}

\newtheorem{remark}[theorem]{Remark}
\numberwithin{equation}{section}

\newcommand{\blankbox}[2]

\begin{document}
\title{Extending structures for Gel'fand-Dorfman bialgebras}
\author{Jiajia Wen}
\address{School of Mathematics, Hangzhou Normal University,
Hangzhou, 311121, China}
\email{2020111008010@stu.hznu.edu.cn}

\author{Yanyong Hong}
\address{School of Mathematics, Hangzhou Normal University,
Hangzhou, 311121, China}
\email{hongyanyong2008@yahoo.com}
\thanks{This work was supported by the National Natural Science Foundation of China (No. 12171129, 11871421), the Zhejiang Provincial Natural Science Foundation of China (No. LY20A010022) and the Scientific Research Foundation of Hangzhou Normal University (No. 2019QDL012).}
\begin{abstract}
Gel'fand-Dorfman bialgebra, which is both a Lie algebra and a Novikov algebra with some compatibility condition, appears in the study of Hamiltonian pairs in  completely integrable systems and a class of special Lie conformal algebras called quadratic Lie conformal algebras. In this paper, we investigate the extending structures problem for Gel'fand-Dorfman bialgebras, which is equivalent to some extending structures problem of quadratic Lie conformal algebras. Explicitly, given a Gel'fand-Dorfman bialgebra $(A, \circ, [\cdot,\cdot])$, this problem asks that how to describe and classify all Gel'fand-Dorfman bialgebraic structures on a vector space $E$ $(A\subset E$) such that $(A, \circ, [\cdot,\cdot])$ is a subalgebra of $E$ up to an isomorphism whose restriction on $A$ is the identity map. Motivated by the theories of extending structures for Lie algebras and Novikov algebras, we construct an object $\mathcal{GH}^2(V,A)$ to answer the extending structures problem by introducing a definition of unified product for Gel'fand-Dorfman bialgebras, where $V$ is a complement of $A$ in $E$. In particular, we investigate the special case when $\text{dim}(V)=1$ in detail.
\end{abstract}
\maketitle
\section{Introduction}
Gel'fand-Dorfman bialgebra (see \cite{Xu2}) $(A, \circ, [\cdot,\cdot])$ is a vector space with a Lie algebraic structure $(A, [\cdot, \cdot])$ and a Novikov algebraic (the term was proposed in \cite{O1}) structure $(A, \circ)$ with some compatibility condition. It corresponds to a Hamiltonian pair, which plays fundamental roles in completely integrable systems (see \cite{GD}). Moreover, there are also close relationships between Gel'fand-Dorfman bialgebras and Lie conformal algebras (see \cite{Xu1, H2, HW}), which describe the algebraic properties of the operator product
expansion in two-dimensional conformal field theory (\cite{K1, DK}). In particular, by the result in \cite{Xu1}, a Gel'fand-Dorfman bialgebra is equivalent to a special Lie conformal algebra named quadratic Lie conformal algebra. Explicitly, the correspondence is given as follows. $R=\mathbb{C}[\partial]V$ is a quadratic Lie conformal algebra, i.e. for all $a$, $b\in V$,
\begin{eqnarray*}
[a_\lambda b]=\partial (b\circ a)+\lambda (b\ast a)+[b,a],
\end{eqnarray*}
where $\circ$, $\ast$, $[\cdot, \cdot]$ are three linear operations from $V\otimes V\rightarrow V$, if and only if $b\ast a=a\circ b+b\circ a$ for all $a$, $b\in V$ and $(V, \circ, [\cdot, \cdot])$ is a Gel'fand-Dorfman bialgebra. Note that both a Lie algebra with the trivial Novikov algebraic structure and a Novikov algebra with the trivial Lie algebraic structure are Gel'fand-Dorfman bialgebra. Therefore, it should be pointed out that Novikov algebras (special Gel'fand-Dorfman bialgebras) also appear in the study of linear Poisson brackets of hydrodynamic type (see \cite{BN}).

There have been some works on Gel'fand-Dorfman bialgebras, for example, the classification of Gel'fand-Dorfman bialgebras in low dimensions (see \cite{BM}), kinds of constructions (see \cite{Xu1, Xu3}),
Gel'fand-Dorfman bialgebras on some infinite-dimensional Lie algebras (see \cite{Xu1,OZ}), the study of Gel'fand-Dorfman bialgebras related with differential Poisson algebras (see \cite{KSO}) and so on.

In this paper, we plan to investigate the following problem of Gel'fand-Dorfman bialgebras:\\
{\bf Extending structures problem:}  {\emph {Given a Gel'fand-Dorfman bialgebra $(A, \circ, [\cdot, \cdot])$ and $E$ a vector space containing $A$. Describe and classify all Gel'fand-Dorfman bialgebraic structures on $E$ such that $(A, \circ, [\cdot,\cdot])$ is a subalgebra of $E$ up to an isomorphism whose restriction on $A$ is the identity map.}\\
This problem is natural and important from the point of view of Gel'fand-Dorfman bialgebras, i.e. how to obtain a `big' Gel'fand-Dorfman bialgebra from a given  `small' Gel'fand-Dorfman bialgebra. This is one motivation. Our another important motivation is from the study of Lie conformal algebras. A general theory of extending structures for Lie conformal algebras was presented in \cite{HS}. Note that quadratic Lie conformal algebras are a class of important Lie conformal algebras including Virasoro conformal algebra, current conformal algebras, Heisenberg-Virasoro conformal algebra (see \cite{SY}) and so on. Therefore, a natural problem appears:\\
\emph{ Given a quadratic Lie conformal algebra $R=\mathbb{C}[\partial]A$, $E=\mathbb{C}[\partial]A\oplus \mathbb{C}[\partial]V$ containing $R$ as a $\mathbb{C}[\partial]$-submodule. Describe and classify all quadratic Lie conformal algebraic structures on $E=\mathbb{C}[\partial](A\oplus V)$ such that $R$ is a Lie conformal subalgebra of $E$ up to an isomorphism whose restriction on $R$ is the identity map.}\\
Note that if $\varphi: A\oplus V\mapsto A\oplus V$ is an isomorphism of Gel'fand-Dorfman bialgebras, then $\Phi:\mathbb{C}[\partial](A\oplus V)\mapsto \mathbb{C}[\partial](A\oplus V)$ defined by $\Phi(p(\partial)a+q(\partial)v)=p(\partial)\varphi(a)+q(\partial)\varphi(v)$ for all $a\in A$, $v\in V$, $p(\partial)$, $q(\partial)\in \mathbb{C}[\partial]$ is an isomorphism of the corresponding quadratic Lie conformal algebras. If we restrict the isomorphisms in the problem above to those induced from the isomorphisms of Gel'fand-Dorfman bialgebras on $A\oplus V$, by the correspondence of Gel'fand-Dorfman bialgebras and quadratic Lie conformal algebras, this problem is equivalent to the extending structures problem of Gel'fand-Dorfman bialgebras. This is the second motivation to investigate this problem. Similar problems for groups, Lie algebras, associative algebras,  Hopf algebras, Novikov algebras and so on have been investigated in \cite{AM1, AM2, AM3, AM4, H1} respectively.
Note that a Gel'fand-Dorfman bialgebra is both a Lie algebra and a Novikov algebra. Motivated by the theories of extending structures for Lie algebras and Novikov algebras developed in \cite{AM2, H1} respectively, we introduce a definition of unified product for Gel'fand-Dorfman bialgebras. With this tool, we construct an object $\mathcal{GH}^2(V,A)$  to give a theoretical answer for the extending structures problem of Gel'fand-Dorfman bialgebras, where $V$ is a complement of $A$ in $E$. Moreover, when $\text{dim} (V)=1$, it is shown that $\mathcal{GH}^2(V,A)$  can be computed and a specific example is given.

This paper is organized as follows. In Section 2, we recall some definitions of Gel'fand-Dorfman bialgebras and some facts about the extending structures of Lie algebras and Novikov algebras. In Section 3, we introduce the definition of unified product for Gel'fand-Dorfman bialgebras and construct an object $\mathcal{GH}^2(V,A)$ to give a theoretical answer for the extending structures problem, where $V$ is a complement of $A$ in $E$. In Section 4, a special case of unified product when $\text{dim}(V)=1$ is investigated in detail and an example to compute $\mathcal{GH}^2(V,A)$  is provided. By this example, we obtain  all quadratic Lie conformal algebraic structures on $\mathbb{C}[\partial](\mathbb{C}L\oplus \mathbb{C}W\oplus \mathbb{C}x)$ containing the Heisenberg-Virasoro conformal algebra $\mathbb{C}[\partial](\mathbb{C}L\oplus \mathbb{C}W)$ as a subalgebra up to the isomorphisms which are induced from the isomorphisms of the corresponding Gel'fand-Dorfman bialgebras whose restrictions on $\mathbb{C}L\oplus \mathbb{C}W$ is the identity map.

Throughout this paper, $K$ is a field and $\mathbb{C}$ is the field of complex
numbers. All vector spaces, Novikov algebras, Lie algebras, Gel'fand-Dorfman bialgebras, linear or bilinear maps are over $K$.
\section{ Preliminaries}
In this section, we recall some definitions about Gel'fand-Dorfman bialgebras and some results about extending structures for Lie algebras and Novikov algebras.

First, let us recall some basic concepts about Lie algebras (see \cite{AM2}).

A \textbf{Lie algebra} is a vector space $\mathfrak{g}$ together with a bilinear map $[\cdot, \cdot]: \mathfrak{g}\times \mathfrak{g}\to \mathfrak{g}$  satisfying the following conditions:
\begin{eqnarray*}
&&[x,x]=0,\\
&&[x,[y,z]]=[[x,y],z]+[y,[x,z]],
\end{eqnarray*}
for all $x$, $y$, $z\in \mathfrak{g}$. We denote it by $(\mathfrak{g},[\cdot, \cdot])$.

Let $(\mathfrak{g}, [\cdot, \cdot])$ be a Lie algebra and $V$ be a vector space. Let $\triangleright: \mathfrak{g}\times V\to V$ be a bilinear map. $(V,\triangleright)$ is called a \textbf{left module} of $(\mathfrak{g},[\cdot, \cdot])$ if
\begin{eqnarray*}
&&[x,y]\triangleright v=x\triangleright(y\triangleright v)-y\triangleright(x\triangleright v),
\end{eqnarray*}
for all $x$, $y\in \mathfrak{g}$, $v\in V$.

Let $\triangleleft: V\times \mathfrak{g}\to V$ be a bilinear map. $(V, \triangleleft)$ is called a \textbf{right  module} of $(\mathfrak{g},[\cdot, \cdot])$ if
\begin{eqnarray*}
&&v\triangleleft[x,y]=(v\triangleleft x)\triangleleft y-(v\triangleleft y)\triangleleft x,
\end{eqnarray*}
for all $x$, $y\in \mathfrak{g}$, $v\in V$.

Naturally, any left module of $(\mathfrak{g},[\cdot, \cdot])$ is a right module of $(\mathfrak{g},[\cdot, \cdot])$ via $x\triangleright v:=-v\triangleleft x$ and vice versa.

Next, we recall some definitions about Novikov algebras. These facts can be found in \cite{H1}.
\begin{definition}
A \textbf{Novikov algebra} is a vector space $A$ together with a bilinear map $\circ : A\times A\to A$ satisfying the following conditions:
\begin{eqnarray*}
&&(a,b,c) = (b,a,c), \text{where the associator $(a,b,c)= (a\circ b)\circ c-a\circ(b\circ c)$,}\\
&&(a\circ b)\circ c=(a\circ c)\circ b,
\end{eqnarray*}
for all $a$, $b$, $c\in A$. We denote it by $(A, \circ)$.
\end{definition}
\begin{remark}
Let $(A, \circ)$ be a Novikov algebra. Then there is a natural Lie algebraic structure on the vector space $A$ as follows:
\begin{eqnarray*}
[a, b]=a\circ b-b\circ a,~~~\text{for all $a$, $b\in A$.}
\end{eqnarray*}
We denote this Lie algebra by $(g(A), [\cdot,\cdot])$ and this Lie algebra is called the {\bf sub-adjacent Lie algebra of $(A, \circ)$}.
\end{remark}
\begin{definition}
Let $(A, \circ)$ be a Novikov algebra, $V$ be a vector space and $l_A$, $r_A: A\to gl(V)$ be two linear maps.  $(V, l_A, r_A)$ is called a {\bf bimodule} of $(A, \circ)$ if
\begin{eqnarray*}
&&l_A(a)l_A(b)v-l_A(a\circ b)v=l_A(b)l_A(a)v-l_A(b\circ a)v,\\
&&l_A(a)r_A(b)v-r_A(b)l_A(a)v=r_A(a\circ b)v-r_A(b)r_A(a)v,\\
&&l_A(a\circ b)v=r_A(b)l_A(a)v,\\
&&r_A(a)r_A(b)v=r_A(b)r_A(a)v,
\end{eqnarray*}
for all $a$, $b\in A$, $v\in V$.
\end{definition}

Then we recall some concepts and some facts about extending structures of Lie algebras and Novikov algebras.
\begin{definition} (see Definition 3.1 and Theorem 3.2 in \cite{AM2})
Let $(\mathfrak{g}, [\cdot,\cdot])$ be a Lie algebra and $V$ be a vector space. A \textbf{Lie extending datum} of $(\mathfrak{g}, [\cdot,\cdot])$ by $V$ is a system $\Omega(\mathfrak{g},V)=(\triangleleft,\triangleright,h,\{\cdot,\cdot\})$ consisting of four bilinear maps $\triangleleft: V\times \mathfrak{g}\to V$, $\triangleright: V\times \mathfrak{g}\to \mathfrak{g}$, $h: V\times V\to \mathfrak{g}$, $\{\cdot,\cdot\}: V\times V\to V$ satisfying the following seven conditions:
\begin{eqnarray*}
(L1)&&h(x,x)=0,\{x,x\}=0,\\
(L2)&&x\triangleleft[a,b]=(x\triangleleft a)\triangleleft b-(x\triangleleft a)\triangleleft b,\\
(L3)&&x\triangleright[a,b]=[x\triangleright a,b]+[a,x\triangleright b]+(x\triangleleft a)\triangleright b-(x\triangleleft b)\triangleright a,\\
(L4)&&\{x,y\}\triangleleft a=\{x,y\triangleleft a\}+\{x\triangleleft a,y\}+x\triangleleft (y\triangleright a)-y\triangleleft (x\triangleright a),\\
(L5)&&\{x,y\}\triangleright a=x\triangleright(y\triangleright a)-y\triangleright(x\triangleright a)+[a,h(x,y)]+h(x,y\triangleleft a)+h(x\triangleleft a,y),\\
(L6)&&h(x,\{y,z\})+h(y,\{z,x\})+h(z,\{x,y\})+x\triangleright h(y,z)+y\triangleright h(z,x)\\
&&+z\triangleright h(x,y)=0,\\
(L7)&&\{x,\{y,z\}\}+\{y,\{z,x\}\}+\{z,\{x,y\}\}+x\triangleleft h(y,z)+y\triangleleft h(z,x)\\
&&+z\triangleleft h(x,y)=0,
\end{eqnarray*}
for all $a$, $b\in \mathfrak{g}$, $x$, $y$, $z\in V$.

Let $\Omega(\mathfrak{g},V)=(\triangleleft,\triangleright,h,\{\cdot,\cdot\})$ be a Lie extending datum of a Lie algebra $(\mathfrak{g}, [\cdot,\cdot])$ by $V$ and let $\mathfrak{g}\natural V$ be the vector space $\mathfrak{g}\times V$ with the bracket $[\cdot,\cdot]$ defined by:
\begin{eqnarray*}
&&[(a,x),(b,y)]=([a,b]+x\triangleright b-y\triangleright a+h(x,y),\{x,y\}+x\triangleleft b-y\triangleleft a),
\end{eqnarray*}
for all $a$, $b \in \mathfrak{g}$, $x, y\in V$. Then $\mathfrak{g}\natural V$ is a Lie algebra called the {\bf unified product} of $\mathfrak{g}$ and $\Omega(\mathfrak{g},V)$.
\end{definition}
\begin{definition}
(see Definition 5.2 in \cite{AM2}) A \textbf{twisted derivation} of a Lie algebra $\mathfrak{g}$ is a pair $(\lambda, D)$ consisting of two linear maps $\lambda: \mathfrak{g}\rightarrow K$ and $D: \mathfrak{g}\rightarrow \mathfrak{g}$ such that for all $a$, $b\in \mathfrak{g}$,
\begin{eqnarray*}
(TD1)&&\lambda([a,b])=0,\\
(TD2)&&D([a,b])=[D(a),b]+[a,D(b)]+\lambda(a)D(b)-\lambda(b)D(a).
\end{eqnarray*}
\end{definition}
\begin{definition} (see Definition 3.1, Theorem 3.2 and Corollary 3.5 in \cite{H1})
Let $(A,\circ)$ be a Novikov algebra and $V$ be a vector space. A \textbf{Novikov extending datum} of $(A, \circ)$ by $V$ is a system $\Omega(A,V)=(l_A, r_A, l_V, r_V, f, \ast)$ consisting of four linear maps and two bilinear maps $l_A, r_A: A\to gl(V)$, $l_V, r_V: V\to gl(A)$, $f: V\times V\to A$, $\ast: V\times V\to V$ satisfying the following conditions:
\begin{eqnarray*}
(N1)&&l_V(x)(a\circ b)=-l_V(l_A(a)x)b +l_V(r_A(a)x)b\\
&&+(l_V(x)a)\circ b-(r_V(x)a)\circ b +r_V(r_A(b)x)a + a\circ(l_V(x)b),\\
(N2)&&l_A(a)r_A(b)x-r_A(b)l_A(a)x=r_A(a\circ b)x-r_A(b)r_A(a)x,\\
(N3)&&r_V(x)(a\circ b)-r_V(x)(b\circ a)=r_V(l_A(b)x)a-r_V(l_A(a)x)b\\
&&+a\circ(r_V(x)b)-b\circ(r_V(x)a),\\
(N4)&&l_A(a\circ b-b\circ a)x=l_A(a)l_A(b)x-l_A(b)l_A(a)x,\\
(N5)&&r_V (x\ast y)a = r_V (y)(r_V (x)a)- r_V (y)( l_V (x)a) + l_V (x)r_V (y)a\\
&&+f(l_A(a)x, y) + f(x, l_A(a)y)- a\circ f(x, y) - f(r_A(a)x, y),\\
(N6)&&l_A(a)(x\ast y) = -l_A(l_V (x)a-r_V (x)a)y + (l_A(a)x-r_A(a)x)\ast y\\
&&+r_A(r_V (y)a)x + x\ast(l_A(a)y), \\
(N7)&&l_V (x\ast y-y\ast x)a = (l_V (x)l_V (y)-l_V (y)l_V (x))a\\
&&-(f(x, y) - f(y, x)) \circ a + f(x, r_A(a)y) - f(y, r_A(a)x),\\
(N8)&&r_A(a)(x\ast y-y\ast x) = r_A(l_V (y)a)x-r_A(l_V (x)a)y \\
&&+x\ast(r_A(a)y)-y\ast(r_A(a)x),\\
(N9)&&f(x\ast y, z) - f(x, y\ast z) - f(y\ast x, z) + f(y, x \ast z)\\
&&+ r_V (z)(f(x, y) - f(y, x)) -l_V (x)f(y, z) + l_V (y)f(x, z) = 0,\\
(N10)&&(x\ast y)\ast z-x\ast(y\ast z)-(y\ast x)\ast z + y\ast(x\ast z)\\
&&+l_A(f(x, y) - f(y, x))z-r_A(f(y, z))x + r_A(f(x, z))y = 0.\\
(N11)&&(l_V (x)a)\circ b + l_V (rA(a)x)b = (l_V (x)b)\circ a + l_V (r_A(b)x)a,\\
(N12)&&r_A(b)r_A(a)x = r_A(a)r_A(b)x,\\
(N13)&&(r_V (x)a)\circ b + l_V (l_A(a)x)b = r_V (x)(a\circ b),\\
(N14)&&r_A(b)l_A(a)x = l_A(a\circ b)x,\\
(N15)&&r_V (y)r_V (x)a + f(l_A(a)x, y) = r_V (x)r_V (y)a + f(l_A(a)y, x),\\
(N16)&&l_A(r_V (x)a)y + (l_A(a)x)\ast y = l_A(r_V (y)a)x + (l_A(a)y)\ast x,\\
(N17)&&r_V (y)(l_V (x)a) + f(r_A(a)x, y) = f(x, y)\circ a + l_V (x\ast y)a,\\
(N18)&&l_A(l_V (x)a)y + (r_A(a)x)\ast y = r_A(a)(x\ast y),\\
(N19)&&r_V (z)f(x, y) + f(x\ast y, z) = r_V (y)f(x, z) + f(x\ast z, y),\\
(N20)&&l_A(f(x, y))z + (x\ast y)\ast z = l_A(f(x, z))y + (x\ast z)\ast y.
\end{eqnarray*}
for all $a$, $b\in A$, $x$, $y$, $z\in V$.

Let $\Omega(A,V)=(l_A,r_A,l_V,r_V,f,\ast)$ be a Novikov extending datum of $(A, \circ)$ by $V$ and $A\natural V$ be the vector space $A\times V$ with the bilinear map $\circ$ defined by:
\begin{eqnarray*}
&&(a,x)\circ(b,y)=(a\circ b+ l_V(x)b+ r_V(y)a+f(x,y), x\ast y+l_A(a)y+r_A(b)x),
\end{eqnarray*}
for all $a$, $b\in A$, $x$, $y\in V$.
Then $A\natural V$ is a Novikov algebra called the {\bf unified product} of $(A, \circ)$ and $\Omega(A,V)$.
\end{definition}

\begin{definition} (see Definition 4.3 in \cite{H1})
Let $(A,\circ)$ be a Novikov algebra. A \textbf{Novikov flag datum} of $(A,\circ)$ is a 6-tuples $(p, q, S, T, a_1, k)$ consisting of four linear maps: $p$, $q: A\to K$, $S$, $T: A\to A$ and two elements $a_1\in A$, $k\in K$ satisfying the following conditions:
\begin{eqnarray*}
(FN1)&&p(a\circ b) = p(b\circ a),~~~q(a\circ b)=q(a)q(b),\\
(FN2)&&S(a\circ b) = S(a)\circ b + a\circ S(b) + (q(a) - p(a))S(b) + q(b)T(a) - T(a)\circ  b,\\
(FN3)&&T(a\circ b) - T(b\circ a) = p(b)T(a) - p(a)T(b) + a\circ  T(b) - b\circ  T(a),\\
(FN4)&&T^2(a) = T(S(a)) - S(T(a)) + a\circ a_1 + (q(a) - 2p(a))a_1 +kT(a),\\
(FN5)&&p(S(a)) - p(T(a)) = q(T(a)) + k(p(a) - q(a)).\\
(FN6)&&S(a)\circ b + q(a)S(b) = S(b)\circ a + q(b)S(a),\\
(FN7)&&T(a\circ b) = T(a)\circ b + p(a)S(b),\\
(FN8)&&p(a\circ b) = p(a)q(b),\\
(FN9)&&T(S(a)) = a_1\circ a +kS(a)-q(a)a_1,\\
(FN10)&&p(S(a)) = 0,
\end{eqnarray*}
for all $a$, $b\in A$.
\end{definition}

\begin{definition} (see Definition 2.5 in \cite{H1})
Let $(A, \circ)$ be a Novikov algebra. A linear map $T:A\rightarrow A$ is called  a \textbf{quasicentroid} if
$T$ satisfies
\begin{eqnarray}
&&T(a\circ b)=T(a)\circ b,\\
&&T(a\circ b)-T(b\circ a)=a\circ T(b)-b\circ T(a), ~~~\text{for all $a$, $b\in A$.}
\end{eqnarray}
\end{definition}

\begin{remark}
For any $b\in A$, there is a quasicentroid $T_b$ associated to it defined by $T_b(a)=a\circ b$ for all $a\in A$. We call $T_b$
an \textbf{inner quasicentroid} of $(A, \circ)$.
\end{remark}

Finally, let us recall some definitions about Gel'fand-Dorfman bialgebras.
\begin{definition} (see \cite{Xu1})
A \textbf{Gel'fand-Dorfman bialgebra} is a vector space $A$ together with two bilinear operations $[\cdot, \cdot]$ and $\circ$ such that $(A,\circ )$ is a Novikov algebra, $(A,[\cdot,\cdot])$ is a Lie algebra, and they satisfy the following condition:
\begin{eqnarray*}
&&[a,b \circ c]-[c,b \circ a]+[b,a] \circ c-[b,c] \circ a-b \circ [a,c]=0,
\end{eqnarray*}
for all $a$, $b$, $c\in A$. Denote it by $(A, \circ, [\cdot,\cdot])$.

Let $(A, \circ, [\cdot, \cdot])$ and $(B, \circ, [\cdot, \cdot])$ be two Gel'fand-Dorfman bialgebras. If a linear map $\varphi: A\rightarrow B$ is both a homomorphism of Lie algebras and a homomorphism of Novikov algebras, then $\varphi$ is called a {\bf homomorphism of Gel'fand-Dorfman bialgebras}.

Let $T\subseteq A$ be a subspace of $A$. If $(T, [\cdot, \cdot])$ is a subalgebra (ideal) of $(A, [\cdot, \cdot])$ and $(T, \circ)$ is a Novikov subalgebra ( a bi-sided ideal) of $(A, \circ)$, then $(T, \circ, [\cdot,\cdot])$ is called a {\bf subalgebra} ({\bf ideal}) of Gel'fand-Dorfman bialgebra $(A, \circ, [\cdot, \cdot])$.
\end{definition}

\begin{definition}
Let $(A, \circ, [\cdot, \cdot])$ be a Gel'fand-Dorfman bialgebra and $V$ be a vector space, together with a bilinear map  $\triangleright $: $A \times V \to V$ and two linear maps $l_A$,$\ r_A$: $A \to gl(V)$. Then $(V,\triangleright,l_A, r_A)$ is called a \textbf{left module} of $(A, \circ, [\cdot, \cdot])$ if $(V,\triangleright)$ is a left module of Lie algebra $(A, [\cdot,\cdot])$, $(V, l_A, r_A)$ is a bimodule of Novikov algebra $(A, \circ)$, and
\begin{eqnarray*}
&&a\triangleright(l_A(b)v)+(b\circ a)\triangleright v+l_A([b,a])v-r_A(a)(b\triangleright v)-l_A(b)(a \triangleright v)=0,\\
&&a\triangleright(r_A(b)v)-b\triangleright (r_A(a)v)-r_A(b)(a\triangleright v)+r_A(a)(b\triangleright v)-r_A([a,b])(v) =0,
\end{eqnarray*}
for all $a$, $b\in A$, $v\in V$.

Similarly,  $(V, \triangleleft, l_A, r_A)$ is called a \textbf{right module} of $(A, \circ, [\cdot, \cdot])$, if $(V,\triangleleft)$ is a right module of Lie algebra $(A, [\cdot,\cdot])$, $(V, l_A, r_A)$ is a bimodule of Novikov algebra $(A, \circ)$, and
\begin{eqnarray*}
&&(l_A(b)v)\triangleleft a+v\triangleleft(b\circ a)-l_A([b,a])v-r_A(a)(v\triangleleft b)-l_A(b)(v\triangleleft a)=0,\\
&&(r_A(b)v)\triangleleft a-(r_A(a)v)\triangleleft b-r_A(b)(v\triangleleft a)+r_A(a)(v\triangleleft b)+r_A([a,b])v=0,
\end{eqnarray*}
for all $a$, $b\in A$, $v\in V$.
\end{definition}
\begin{definition}
Let $(A, \circ, [\cdot, \cdot])$ be a Gel'fand-Dorfman bialgebra, $E$ be a vector space such that $A$ is a subspace of $E$. Let $(E, \circ, [\cdot,\cdot])$ and $(E, \circ^{'}, [\cdot,\cdot]^{'})$ be two Gel'fand-Dorfman bialgebraic structures on $E$ containing $(A, \circ, [\cdot, \cdot])$ as a subalgebra. If there is an isomorphism of Gel'fand-Dorfman bialgebras $\varphi: (E, \circ, [\cdot,\cdot])\rightarrow (E, \circ^{'}, [\cdot,\cdot]^{'})$ whose restriction on $A$ is the identity map, then we call that $(E, \circ, [\cdot,\cdot])$ and $(E, \circ^{'}, [\cdot,\cdot]^{'})$ are {\bf equivalent}. We denote it by $(E, \circ , [\cdot,\cdot] )\equiv ( E, \circ^{'}, [\cdot,\cdot]^{'})$.
\end{definition}

It is easy to see that $\equiv$ is an equivalence relation on the set of all Gel'fand-Dorfman bialgebraic structures on $E$ containing $(A,\circ, [\cdot,\cdot])$ as a subalgebra. Denote the set of all equivalence classes via $\equiv$ by $\text{Extd}(E,A)$. Therefore, for studying the extending structures problem, we only need to characterize $\text{Extd}(E,A)$.

\section{unified products for Gel'fand-Dorfman bialgebras}
In this section, we introduce a definition of unified product for Gel'fand-Dorfman bialgebras. Using this tool, we give a characterization of $\text{Extd}(E,A)$.
\begin{definition}
Let $(A,\circ,[\cdot,\cdot])$ be a Gel'fand-Dorfman bialgebra and $V$ a vector space. An {\bf extending datum }of $(A,\circ,[\cdot,\cdot])$  by $V$ is a system $\Omega(A,V) = (l_A, r_A, l_V, r_V, f, \ast, \triangleleft, \triangleright, h, \{\cdot,\cdot\})$ consisting of four linear maps and six bilinear maps as follows:
\begin{eqnarray*}
&&f: V \times V \to A,~~~\ast: V \times V \to V,\\
&&l_A, r_A: A \to gl(V),~~~l_V, r_V: V \to gl(A),\\
&&\triangleleft: V \times A \to V, ~~~\triangleright: V \times A \to A,~~~ h: V \times V \to A, ~~~\{\cdot,\cdot\}: V \times V \to V.
\end{eqnarray*}
Let $\Omega(A, V) = (l_A, r_A, l_V, r_V, f, \ast, \triangleleft, \triangleright, h, \{\cdot,\cdot\})$ be an extending datum of a Gel'fand-Dorfman bialgebra $A$ by a vector space $V$. We denote by $A \natural_{\Omega(A,V)} V = A \natural V$ the vector space $A \times V$ together with the bilinear operations $ \circ $ and $[\cdot,\cdot]$ defined by:
\begin{eqnarray*}
&&(a,x) \circ (b,y)=(a \circ b+ l_V(x) b+ r_V(y)a+f(x,y),x*y+ l_A(a) y+ r_A(b)x),\\
&&[(a,x),(b,y)]=([a,b]+x\triangleright b-y\triangleright a+h(x,y),\{x,y\}+x\triangleleft b-y\triangleleft a),
\end{eqnarray*}
for all $a$, $b\in A$, $x$, $y \in V$. The object $A \natural V$ is called the {\bf unified product} of $(A, \circ, [\cdot,\cdot])$ and $\Omega(A, V)$ if it is a Gel'fand-Dorfman bialgebra with the products given above. In this case, the extending datum $\Omega(A, V)$ is called a {\bf Gel'fand-Dorfman extending structure of $(A, \circ, [\cdot,\cdot])$ by $V$}. We denote by  $\mathcal{GD}(A,V)$ the set of all Gel'fand-Dorfman extending structures of $(A, \circ, [\cdot,\cdot])$ by $V$.
\end{definition}
\begin{theorem}\label{t1}
Let $(A, \circ, [\cdot, \cdot])$ be a Gel'fand-Dorfman bialgebra, $V$ be a vector space and $\Omega(A, V)= (l_A, r_A, l_V, r_V, f,\ast,\triangleleft, \triangleright, h, \{\cdot,\cdot\})$ an extending datum of $(A, \circ, [\cdot, \cdot])$ by $V$. Then $A \natural V$ is a unified product if and only if the following conditions hold:
\begin{eqnarray*}
(G0)&&\text{$(l_A, r_A, l_V, r_V, f, \ast)$ is a Novikov extending datum of $(A, \circ)$ by $V$,} \\
&&
\text{$(\triangleleft,\triangleright,h,\{\cdot,\cdot\})$ is a Lie extending datum of $(A, [\cdot,\cdot])$ by $V$,}\\
(G1)&&[a,r_V(x)b]-(l_A(b)x)\triangleright a-x\triangleright (b\circ a)+r_V(x)([b,a])+(x\triangleright b)\circ a\\
&&+l_V(x\triangleleft b)a+b\circ (x\triangleright a)+r_V(x \triangleleft a)b=0,\\
(G2)&&(l_A(b)x)\triangleleft a+x\triangleleft (b\circ a)-l_A([b,a])x-r_A(a)(x\triangleleft b)-l_A(b)(x\triangleleft a)=0,\\
(G3)&&[a,l_V(x)b]-(r_A(b)x)\triangleright a-[b,l_V(x)a]+(r_A(a)x)\triangleright b\\
&&+(x\triangleright a)\circ b+l_V(x\triangleleft a)b-(x\triangleright b)\circ a-l_V(x\triangleleft b)a-l_V(x)([a,b])=0,\\
(G4)&&(r_A(b)x)\triangleleft a-(r_A(a)x)\triangleleft b-r_A(b)
(x\triangleleft a)+r_A(a)(x\triangleleft b)+r_A([a,b])x=0,\\
(G5)&&[a,f(x,y)]-(x\ast y)\triangleright a-y\triangleright (l_V(x)a)-h(y,r_A(a)x)+r_V(y)(x\triangleright a)\\
&&+f(x\triangleleft a,y)-h(x,y)\circ a-l_V(\{x,y\})a+l_V(x)(y\triangleright a)+f(x,y\triangleleft a)=0,\\
(G6)&&(x\ast y)\triangleleft a+\{y,r_A(a)x\}+y\triangleleft (l_V(x)a)-(x\triangleleft a)\ast y-l_A(x\triangleright a)y\\
&&+r_A(a)(\{x,y\})-x\ast (y\triangleleft a)-r_A(y\triangleright a)x=0,\\
(G7)&&x\triangleright (r_V(y)a)+h(x,l_A(a)y)-y\triangleright (r_V(x)a)-h(y,l_A(a)x)-r_V(y)(x\triangleright a)\\
&&-f(x\triangleleft a,y)+r_V(x)(y\triangleright a)+f(y\triangleleft a,x)-a\circ h(x,y)-r_V(\{x,y\})a=0,\\
(G8)&&\{x,l_A(a)y\}+x\triangleleft (r_V(y)a)-\{y,l_A(a)x\}-y\triangleleft(r_V(x)a)-(x\triangleleft a)\ast y\\
&&-l_A(x\triangleright a)y+(y\triangleleft a)\ast x+l_A(y\triangleright a)x-l_A(a)\{x,y\}=0,\\
(G9)&&x\triangleright f(y,z)+h(x,y\ast z)-z\triangleright f(y,x)\\
&&-h(z,y\ast x)+r_V(z)(h(y,x))+f(\{y,x\},z)\\
&&-r_V(x)(h(y,z))-f(\{y,z\},x)-l_V(y)(h(x,z))-f(y,\{x,z\})=0,\\
(G10)&&\{ x,y\ast z\}+x\triangleleft f(y,z)+\{ z,y\ast x\}-z\triangleleft f(y,x)+\{y,x\}\ast z\\
&&+l_A(h(y,x))z-\{y,z\}\ast x-l_A(h(y,z))x-y\ast \{x,z\}-r_A(h(x,z))y=0,
\end{eqnarray*}
for all $a$, $b\in A$, $x$, $y$, $z\in V$.
\end{theorem}
\begin{proof}
By Theorem 3.2 in \cite{AM2} and Theorem 3.2 in \cite{H1},  $(A\natural V, [\cdot,\cdot])$ is a Lie algebra if and only if $(\triangleleft, \triangleright, h, \{\cdot,\cdot\})$ is a Lie extending datum of $(A, [\cdot,\cdot])$ by $V$, and $(A\natural V, \circ)$ is a Novikov algebra if and only if $(l_A, r_A, l_V, r_V, f, \ast)$ is a Novikov extending datum of $(A, \circ)$ by $V$. Then we only need to check that
\begin{eqnarray}
&&\label{eq1}[(a, x) , (b, y)\circ(c, z)]-[(c, z),(b, y)\circ(a, x)]+[(b, y),(a, x)]\circ(c, z)\\
&&-[(b, y),(c, z)]\circ(a, x)-(b, y)\circ[(a, x),(c, z)]=0\nonumber
\end{eqnarray}
for all $a$, $b$, $c\in A$, $x$, $y$, $z\in V$ if and only if $(G1)$-$(G10)$ hold.

Note that $(a, x) = (a, 0) + (0, x)$ in $A\times V$ and (\ref{eq1}) also holds if $((a,x), (b,y), (c,z))$ is changed to $((c,z), (b,y), (a,x))$. Thus, (\ref{eq1}) holds for all $a$, $b$, $c\in A$, $x$, $y$, $z\in V$ if and only if it holds for all triples $((a, 0), (b, 0), (c, 0))$, $((a, 0), (b, 0), (0, x))$, $((a, 0), (0, x), (b, 0))$, $((a, 0), (0, x), (0, y))$, $((0, x), (a, 0),$ $ (0, y))$, and $((0, x), (0, y)$, $(0, z))$, where $a$, $b$, $c\in A$, $x$, $y$, $z\in V$. Obviously, (\ref{eq1}) holds for the triple $((a, 0), (b, 0), (c, 0))$, since $(A, \circ, [\cdot,\cdot])$ is a Gel'fand-Dorfman bialgebra. Then by some computations, we can obtain that (\ref{eq1}) holds for the triple $((a, 0), (b, 0), (0, x))$ if and only if $(G1)$ and $(G2)$ hold, (\ref{eq1}) holds for the triple $((a, 0), (0, x), (b, 0))$ if and only if $(G3)$ and $(G4)$ hold, (\ref{eq1}) holds for the triple $((a, 0), (0, x), (0, y))$ if and only if $(G5)$ and $(G6)$ hold, (\ref{eq1}) holds for the triple $((0, x), (a, 0), (0, y))$ if and only if $(G7)$ and $(G8)$ hold, and (\ref{eq1}) holds for the triple $((0, x), (0, y), (0, z))$ if and only if $(G9)$ and $(G10)$ hold. Then the proof is finished.
\end{proof}

\begin{remark}
In fact, $(G0)$, $(G2)$ and $(G4)$ mean that $(V, \triangleleft, l_A, r_A)$ is a right module of $(A, \circ, [\cdot,\cdot])$.
\end{remark}

\begin{example}\label{examp1}
Let $\Omega(A, V)= (l_A, r_A, l_V, r_V, f, \ast, \triangleleft, \triangleright, h, \{\cdot,\cdot\})$ be a Gel'fand-Dorfman extending datum of $(A, \circ, [\cdot,\cdot])$ by a vector space $V$, where $l_A$, $r_A$ and $\triangleleft$ are trivial maps. For the sake of simplicity, we denote this extending datum by $\Omega(A, V)= (l_V, r_V, f, \ast, \triangleright, h, \{\cdot,\cdot\})$. Then $\Omega(A, V)$ is a Gel'fand-Dorfman extending structure of  $(A, \circ, [\cdot,\cdot])$ by $V$ if and only if $(l_V, r_V, f, \ast)$ is a Novikov extending structure of $(A, \circ)$ by $V$, $ (\triangleright, h, \{\cdot,\cdot\})$ is a Lie extending structure of $(A, [\cdot,\cdot])$ by $V$, $(V,\ast, \{\cdot,\cdot\})$ is a Gel'fand-Dorfman bialgebra and the following conditions hold:
\begin{eqnarray*}
&&[a,r_V(x)b]-x\triangleright(b\circ a)+r_V(x)[b,a]+(x\triangleright b)\circ a+b\circ (x\triangleright a)=0,\\
&&[a,l_V(x)b]-[b,l_V(x)a]+(x\triangleright a)\circ b-(x\triangleright b)\circ a-l_V(x)[a,b]=0,\\
&&[a,f(x,y)]-(x\ast y)\triangleright a-y\triangleright(l_V(x)a)+r_V(y)(x\triangleright a)-h(x,y)\circ a\\
&&-l_V(\{x,y\})a+l_V(x)(y\triangleright a)=0,\\
&&x\triangleright (r_V(y)a)-y\triangleright(r_V(x)a)-r_V(y)(x\triangleright a)+r_V(x)(y\triangleright a)\\
&&-a\circ h(x,y)-r_V(\{x,y\})a=0,\\
&&x\triangleright f(y,z)+h(x,y\ast z)-z\triangleright f(y,x)-h(z,y\ast x)+r_V(z)(h(y,x))\\
&&+f(\{y,x\},z)-r_V(x)(h(y,z))-f(\{y,z\},x)\\
&&-l_V(y)(h(x,z))-f(y,\{x,z\})=0,
\end{eqnarray*}
for all $a$, $b\in A$, $x$, $y$, $z\in V$. In this case, the associated unified product $A\natural V$ denoted by $A\diamond V$ is called the {\bf crossed product} of $(A, \circ, [\cdot,\cdot])$ and $(V, \ast, \{\cdot, \cdot\})$. The crossed product associated to $(A, V, l_V, r_V, f, \triangleright, h)$ satisfying the compatibility conditions above is the Gel'fand-Dorfman bialgebra defined as follows:
\begin{eqnarray*}
&&(a,x)\circ(b,y)=(a\circ b+ l_V(x)b+ r_V(y)a+f(x,y),x\ast y),\\
&&[(a,x),(b,y)]=([a,b]+x\triangleright b-y\triangleright a+h(x,y),\{x,y\}),
\end{eqnarray*}
for all $a$, $b\in A$, $x$, $y\in V$.

Note that $(A, \circ, [\cdot,\cdot])$ is an ideal of $(A\diamond V, \circ, [\cdot, \cdot])$, since $[(a,0),(b,y)]=([a,b]-y\triangleright a,0)$ , $(a,0)\circ(b,y)=(a\circ b+ r_V(y)a, 0)$ and $(a,x)\circ(b,0)=(a\circ b+ l_V(x)b,0)$.
\end{example}
\begin{example}\label{examp2}
Let $\Omega(A, V)= (l_A, r_A, l_V, r_V, f, \ast, \triangleleft, \triangleright, h, $ $\{\cdot,\cdot\})$ be a Gel'fand-Dorfman extending datum of $(A, \circ, [\cdot,\cdot])$ by a vector space $V$, where $f$ and $h$ are trivial maps. We denote this extending datum by $\Omega(A, V)= (l_A, r_A, l_V, r_V, \ast, \triangleleft, \triangleright$, $\{\cdot,\cdot\})$. Then $\Omega(A, V)$ is a Gel'fand-Dorfman extending structure of $(A, \circ, [\cdot,\cdot])$ by $V$ if and only if $(l_A, r_A, l_V, r_V, \ast)$ is a Novikov extending structure of $(A, \circ)$ by $V$, $ (\triangleleft, \triangleright, \{\cdot,\cdot\})$ is a Lie extending structure of $(A, [\cdot,\cdot])$ by $V$, $(V,\ast, \{\cdot,\cdot\})$ is a Gel'fand-Dorfman bialgebra, $(V, \triangleleft, l_A, r_A)$ is a right module of $(A, \circ, [\cdot,\cdot])$, $(A, \triangleright, l_V, r_V)$ is a left module of $(V,\ast,\{\cdot,\cdot\})$ and they satisfy  $(G1)$, $(G3)$, $(G6)$ and $(G8)$.
In this case, the associated unified product $A\natural V$ denoted by $A\bowtie V$ is called the {\bf bicrossed product} of Gel'fand-Dorfman bialgebras $(A, \circ, [\cdot, \cdot])$ and $(V, \ast, \{\cdot,\cdot\})$ associated with the matched pair $(l_A, r_A, \triangleleft, l_V, r_V, \triangleright)$. The bicrossed product associated to $(l_A, r_A, \triangleleft, l_V, r_V, \triangleright)$ satisfying the compatibility conditions above is the Gel'fand-Dorfman bialgebra defined as follows:
\begin{eqnarray*}
&&(a,x)\circ(b,y)=(a\circ b+ l_V(x)b+ r_V(y)a, x\ast y+ l_A(a)y+ r_A(b)x),\\
&&[(a,x),(b,y)]=([a,b]+x\triangleright b-y\triangleright a, \{x,y\}+x\triangleleft b-y\triangleleft a),
\end{eqnarray*}
for all $a$, $b\in A$, $x$, $y\in V$.
Note that $(A, \circ, [\cdot, \cdot])$ and $(V, \ast, \{\cdot,\cdot\})$ are both subalgebras of $A\bowtie V$, since $(a,0)\circ(b,0)=(a\circ b, 0)$, $[(a,0),(b,0)]=([a,b],0)$, $(0,x)\circ(0,y)=(0, x\ast y)$ and $[(0,x), (0,y)]=(0,\{x,y\})$.
\end{example}
\begin{theorem}\label{tt2}
Let $(A, \circ, [\cdot,\cdot])$ be a Gel'fand-Dorfman bialgebra, $E$ a vector space containing $A$ as a subspace and $(E, \circ, [\cdot,\cdot])$ a Gel'fand-Dorfman bialgebraic structure on $E$ such that $(A, \circ, [\cdot,\cdot])$ is a subalgebra of $(E, \circ, [\cdot,\cdot])$. Then there exists a Gel'fand-Dorfman extending structure $\Omega(A,V)=(l_A, r_A, l_V, r_V, f, \ast, \triangleleft, \triangleright, h, \{\cdot,\cdot\})$ of $(A, \circ$, $[\cdot,\cdot])$ by a subspace $V$ of $E$ such that $( E, \circ, [\cdot,\cdot]) \cong A\natural V$ which is an isomorphism of Gel'fand-Dorfman bialgebras whose restriction on $A$ is the identity map.
\end{theorem}
\begin{proof}
Obviously, there exists a natural linear map $p: E \to A$ such that $p(a)=a$ for all $a\in A$. Let $V= Ker(p)$ which is a subspace of $E$ and a complement of $A$ in $E$. Then we define the extending datum of $(A, \circ, [\cdot,\cdot])$ by $V$ as follows:
\begin{eqnarray*}
&&l_A: A \to gl(V),~~~~ l_A (a)x:=a\circ x- p(a\circ x),\\
&&r_A: A \to gl(V), ~~~~r_A(a)x:=x\circ a- p(x\circ a),\\
&&l_V: V \to gl(A), ~~~~l_V(x)a:= p(x\circ a),\\
&&r_V: V \to gl(A), ~~~~r_V(x)a:= p(a\circ x),\\
&&f: V \times V \to A,~~~~ f(x,y):=p(x\circ y),\\
&&\ast: V \times V \to V, ~~~~x\ast y:= x\circ y- p(x\circ y),\\
&&\triangleleft: V \times A \to V,~~~~ x\triangleleft a:= [x,a]-p([x,a]),\\
&&\triangleright: V \times A \to A,~~~~ x\triangleright a:=p([x,a]),\\
&&h: V \times V \to A,~~~~ h(x,y):=p([x,y]),\\
&&\{\cdot,\cdot\}: V\times V \to V,~~~~ \{x,y\}:=[x,y]-p([x,y]),
\end{eqnarray*}
for all $a \in A$, $x, y\in V$.
By referring to the proof in Theorem 3.4 in \cite{AM2} and the proof in Theorem 3.9 in \cite{H1}, it is easy to prove that $\Omega(A, V) = (l_A, r_A, l_V, r_V, f, \ast, \triangleleft, \triangleright, h,$ $ \{\cdot,\cdot\})$ is a Gel'fand-Dorfman extending structures of $(A, \circ, [\cdot,\cdot])$ by $V$ and the linear map $\varphi:  A\natural V\to(E,\circ, [\cdot,\cdot])$, where $\varphi(a,x):=a+x$, is an isomorphism of Gel'fand-Dorfman bialgebras whose restriction on $A$ is the identity map.
\end{proof}

\begin{remark}
By Theorem \ref{tt2} and Example \ref{examp1}, the crossed products $A\diamond V$ describe all Gel'fand-Dorfman bialgebraic structures on the vector space $E=A\oplus V$ such that $(A, \circ, [\cdot,\cdot])$ is an ideal of $E$, where $V$ is a complement of $A$ in $E$.

Similarly, by Theorem \ref{tt2} and Example \ref{examp2}, the bicrossed products $A\bowtie V$ describe all Gel'fand-Dorfman algebraic structures on the vector space $E=A\oplus V$ such that $E$ contains $(A, \circ, [\cdot,\cdot])$ and $(V, \ast, [\cdot,\cdot])$ as two subalgebras.
\end{remark}

By Theorem \ref{tt2}, in order to classify all Gel'fand-Dorfman bialgebraic structures on $E$ containing $A$ as a subalgebra up to an isomorphism whose restriction on $A$ is the identity map, we only need to investigate the isomorphisms between different unified products whose restrictions on $A$ are the identity map.

\begin{lemma}\label{lemm1}
Let $(A, \circ, [\cdot,\cdot])$ be a Gel'fand-Dorfman bialgebra, $V$ a vector space, $\Omega(A, V) = (l_A, r_A, l_V, r_V, f, \ast, \triangleleft, \triangleright, h, \{\cdot,\cdot\})$ and $\Omega^{'}(A, V) = (l_A^{'},r_A^{'},l_V^{'},r_V^{'},f^{'},\ast^{'}$, $\triangleleft^{'},\triangleright^{'},h^{'},\{\cdot,\cdot\}^{'})$ be two Gel'fand-Dorfman extending structures of $(A, \circ, [\cdot,\cdot])$ by $V$. Set $A \natural V$ and $A \natural^{'} V$ be the unified products associated to the extending structures $\Omega(A, V) $ and $ \Omega^{'} (A, V)$ respectively. Then there exists a bijection with respect to the correspondence $\varphi_{\lambda, \mu}(a,x)=(a+ \lambda(x), \mu(x))$ between the set of all isomorphisms
 of Gel'fand-Dorfman bialgebras $\varphi: A \natural V \to A \natural^{'} V $ whose restriction on $A$ is the identity map and the set of pairs $(\lambda, \mu)$, where $\lambda: V\rightarrow A$ is a linear map, $\mu: V\rightarrow V$ is a linear isomorphism and they satisfy the following compatibility conditions:
\begin{eqnarray*}
(D1)&&r_A(a)x=\mu^{-1}(r_A^{'}(a) \mu(x)),\\
(D2)&&l_A(a)x=\mu^{-1}(l_A^{'}(a) \mu(x)),\\
(D3)&&l_V(x)a=\lambda(x)\circ a+l_V^{'}(\mu(x))a-\lambda\mu^{-1}(r_A^{'}(a) \mu(x)),\\
(D4)&&r_V(x)a=a\circ\lambda(x)+ r_V^{'}(\mu(x))a-\lambda\mu^{-1}(l_A^{'}(a) \mu(x)),\\
(D5)&&x\ast y=\mu^{-1}(\mu(x)\ast^{'}\mu(y))+ \mu^{-1}(l_A^{'}(\lambda(x)) \mu(y))+\mu^{-1}(r_A^{'}(\lambda(y)) \mu(x)),\\
(D6)&&f(x,y)=\lambda(x)\circ \lambda(y)+ l_V^{'}( \mu(x)) \lambda(y)+ r_V^{'}(\mu (y)) \lambda (x)+f^{'}(\mu(x), \mu(y))\\
&&- \lambda\mu^{-1}(\mu(x)\ast^{'} \mu(y))- \lambda\mu^{-1}(l_A^{'}(\lambda(x))\mu(y))- \lambda\mu^{-1}(r_A^{'}(\lambda(y)) \mu(x)),\\
(D7)&&x\triangleleft a=\mu^{-1}(\mu(x)\triangleleft^{'}a),\\
(D8)&&x\triangleright a=[\lambda(x),a]+ \mu(x)\triangleright^{'}a-\lambda\mu^{-1}(\mu(x)\triangleleft^{'}a),\\
(D9)&&h(x,y)=[ \lambda(x), \lambda(y)]+ \mu(x)\triangleright^{'} \lambda(y)- \mu(y)\triangleright^{'}\lambda(x)+h^{'} (\mu(x), \mu(y))\\
&&-\lambda\mu^{-1}(\{\mu(x), \mu(y)\}^{'})- \lambda\mu^{-1}(\mu(x)\triangleleft^{'}\lambda(y))+ \lambda\mu^{-1}(\mu(y)\triangleleft^{'}\lambda(x)),\\
(D10)&&\{x,y\}=\mu^{-1}(\{\mu(x), \mu(y)\}^{'})+ \mu^{-1}(\mu(x)\triangleleft^{'}\lambda(y))- \mu^{-1}(\mu(y)\triangleleft^{'}\lambda(x)),\\
\end{eqnarray*}
for all $a \in A$, $x$, $y\in V$.
\end{lemma}
\begin{proof}
Since $\varphi: A \natural V \to A \natural^{'} V $ is a homomorphism of Gel'fand-Dorfman bialgebras whose restriction on $A$ is the identity map, we can set that $\varphi_{\lambda, \mu}(a,x)=(a+ \lambda(x), \mu(x))$, where $\lambda: V\rightarrow A$  and $\mu: V\rightarrow V$ are two linear maps.
Note that a homomorphism of Gel'fand-Dorfman bialgebras is also a homomorphism of Lie algebras and  a homomorphism of Novikov algebras. Therefore this lemma follows directly from  Lemma 3.5 in \cite{AM2} and Lemma 3.12 in \cite{H1}.
\end{proof}
\begin{definition}
Let $(A, \circ, [\cdot,\cdot])$ be a Gel'fand-Dorfman bialgebra and $V$ a vector space. Two Gelf'and-Dorfman extending structures $\Omega(A, V) = (l_A, r_A, l_V, r_V, f, \ast, \triangleleft, $ $\triangleright, h, \{\cdot,\cdot\})$ and $\Omega^{'}(A, V)= (l_A^{'}, r_A^{'}, l_V^{'}, r_V^{'}, f^{'}, \ast^{'}, \triangleleft^{'}, \triangleright^{'}, h^{'}, \{\cdot,\cdot\}^{'})$ are called {\bf equivalent}, which is denoted by $\Omega(A, V)\equiv \Omega^{'}(A, V)$, if there exists a pair of linear maps $(\lambda, \mu)$, where $\lambda: V\to A$, $\mu\in Aut_K(V)$, such that $\Omega(A, V)$ can be obtained  from $\Omega'(A, V)$ by $(\lambda,\mu) $ as $(D1)$-$(D10)$ in Lemma \ref{lemm1}.
\end{definition}
\begin{theorem}\label{t2}
Let $(A, \circ, [\cdot,\cdot])$ be a Gel'fand-Dorfman bialgebra and $E$ a vector space which contains $A$ as a subspace and $V$ a complement of $A$ in $E$.  Denote $\mathcal{GH}^2(V,A):= \mathcal{GD}(A,V)/ \equiv$. Then the map
\begin{eqnarray}
\mathcal{GH}^2(V,A) \mapsto Extd(E,A),~~~~\overline{\Omega(A, V)} \to (A\natural V, \circ ,[\cdot,\cdot])
\end{eqnarray}
is bijective, where $\overline{\Omega(A, V)}$ is the equivalence class of $\Omega(A, V)$ via $\equiv$.
\end{theorem}
\begin{proof}
This conclusion follows directly from  Theorem \ref{t1}, Theorem \ref{tt2} and Lemma \ref{lemm1}.
\end{proof}
\section{Unified products $A\natural V$ when $\text{dim}(V)=1$}
In this section, using the general theory of extending structures for Gel'fand-Dorfman bialgebras developed in Section 3, we investigate the unified products $A\natural V$ in detail when $\text{dim}(V)=1$.

\begin{definition}
Let $(A, \circ, [\cdot,\cdot])$ be a Gel'fand-Dorfman bialgebra. A {\bf Gel'fand-Dorfman flag datum} of $(A, \circ, [\cdot, \cdot])$ is a $8$-tuples $(p, q, S, T, a_1, k, \eta , D)$ consisting of six linear maps: $p$, $q$, $\eta: A\to K $, $D$, $S$, $T: A \to A$ and two elements $a_1\in A$, $k\in K$ satisfying the following conditions:
\begin{eqnarray*}
(GF0)&&(p, q, S, T, a_1, k)~~\text{is  a flag datum of the Novikov algebra $(A,\circ)$}\\
&&\text{and $(\eta, D)$ is a twisted derivation of the Lie algebra $(A, [\cdot,\cdot])$},\\
(GF1)&&[a,T(b)]-p(b)D(a)-D(b\circ a)+T([b,a])+D(b)\circ a+\eta(b)S(a)\\
&&+b\circ D(a)+ \eta(a)T(b)=0,\\
(GF2)&&p([b,a])+ \eta(b)q(a)- \eta(b\circ a)=0,\\
(GF3)&&[a,S(b)]-q(b)D(a)-[b,S(a)]+q(a)D(b)+D(a)\circ b+\eta(a)S(b)\\
&&-D(b)\circ a-\eta(b)S(a)-S([a,b])=0,\\
(GF4)&&q([a,b])=0,\\
(GF5)&&[a,a_1]-kD(a)-D(S(a))+T(D(a))+2\eta(a)a_1+S(D(a))=0,\\
(GF6)&&k\eta(a)-\eta(S(a))+p(D(a))+q(D(a))=0,
\end{eqnarray*}
for all $a$, $b\in A$. Denote by $\mathcal{F}(A)$ the set of all flag datums of $(A, \circ, [\cdot,\cdot])$.
\end{definition}

\begin{proposition}\label{pro1}
Let $(A, \circ, [\cdot,\cdot])$ be a Gel'fand-Dorfman bialgebra and $V$ a vector space of dimension $1$ with a basis $\{x\}$. Then there exists a bijection between the set $\mathcal{GD}(A,V)$ of all Gel'fand-Dorfman  extending structures of $(A, \circ, [\cdot,\cdot])$ by $V$ and $\mathcal{F}(A)$ of all Gel'fand-Dorfman flag datums of $(A, \circ, [\cdot,\cdot])$.
\end{proposition}
\begin{proof}
Let $\Omega (A, V) = (l_A, r_A, l_V, r_V, f, \ast, \triangleleft, \triangleright, h, \{\cdot,\cdot\})$ be a Gel'fand-Dorfman extending structure of $(A, \circ, [\cdot,\cdot])$ by $V$.
Since $\text{dim}(V)=1$, $h=0$, $\{\cdot,\cdot\}=0$, and we can set
\begin{eqnarray*}
&&l_A(a)x=p(a)x,~~~r_A(a)x=q(a)x,~~~l_V(x)a=S(a),~~r_V(x)a=T(a), \\
&& f(x,x)=a_1,~~x\ast x=kx,~~x\triangleleft a=\eta(a)x,~~x\triangleright a=D(a),
\end{eqnarray*}
where $a_1\in A$, $k\in K$, and $p$, $q$, $\eta: A\to K $, $D$, $S$, $T: A \to A$ are six linear maps. Take them into $(G0)$-$(G10)$. Then by Proposition 4.4 in \cite{H1}, Proposition 5.4 in \cite{AM2} and a straightforward computation, we obtain that $\Omega (A, V) = (l_A, r_A, l_V, r_V, f, \ast, \triangleleft, \triangleright, h, \{\cdot,\cdot\})$ is a Gel'fand-Dorfman extending structure if and only if $(GF0)$-$(GF6)$ hold. Then the proof is finished.
\end{proof}

We denote the unified product associated to the Gel'fand-Dorfman extending structure corresponding to a Gel'fand-Dorfman flag datum $(p, q, S, T, a_1, k, \eta , D)$ by $GD(A, x\mid p, q, S, T, a_1, k, \eta , D)$.
\begin{theorem}\label{t3}
Let $(A, \circ, [\cdot,\cdot])$ be a Gel'fand-Dorfman bialgebra of codimension $1$ in the vector space $E$. Then there exists a bijection
\begin{eqnarray*}
&&Extd(E,A)\cong \mathcal{GH}^2(K,A)\cong \mathcal{F}(A)/\equiv,
\end{eqnarray*}
where $\equiv$ is the equivalence relation on the set $\mathcal{F}(A)$ of all Gel'fand-Dorfman flag datums of $(A, \circ, [\cdot, \cdot])$ defined as follows: $(p, q, S, T, a_1, k, \eta, D)\equiv (p^{'}, q^{'}, S^{'}, T^{'}, a_{1}^{'}, k^{'}, \eta^{'}, D^{'})$ if and only if $p=p^{'}$, $q=q^{'}$, $\eta=\eta^{'}$ and there exists a pair $(a_0, \beta) \in A\times K^\ast $ satisfying the following conditions:
\begin{eqnarray*}
(FD1)&&S(a)=a_0\circ a+\beta S^{'}(a)-q(a)a_0,\\
(FD2)&&T(a)=a\circ a_0+\beta T^{'}(a)-p(a)a_0,\\
(FD3)&&a_1=a_0\circ a_0+\beta S^{'}(a_0)+\beta T^{'}(a_0)+\beta^2a_{1}^{'}-ka_0,\\
(FD4)&&k=\beta k^{'}+p(a_0)+q(a_0),\\
(FD5)&&D(a)=\beta D^{'}(a)+[a_0,a]- \eta(a)a_0,
\end{eqnarray*}
for all $a\in A$. The bijection between $\mathcal{F}(A)/\equiv$ and \text{Extd} $(E,A)$ is given by $
\overline{(p,q,S,T,a_1, k,\eta,D)} \to GD(A, x\mid p, q, S, T, a_1, k, \eta , D)$, where $\overline{(p, q, S, T, a_1, k, \eta, D)}$ is the equivalence class of $(p, q, S, T, a_1, k, \eta, D)$ by $\equiv$.
\end{theorem}
\begin{proof}
 Let $(p, q, S, T, a_1, k, \eta, D)$, $(p^{'}, q^{'}, S^{'}, T^{'}, a_{1}^{'}, k^{'}, \eta^{'}, D^{'})\in \mathcal{F}(A)$ and $\Omega(A, V)$, $\Omega'(A, V)$ be the corresponding Gel'fand-Dorfman extending structures respectively. Since $\text{dim}(V)=1$, assume that $x$ is a basis of $V$. Then we can set $\lambda$, $\mu$ in Lemma \ref{lemm1} as follows:
 \begin{eqnarray*}
 \lambda(x)=a_0,~~~\mu(x)=\beta x,
 \end{eqnarray*}
 where $a_0\in A$ and $\beta\in K^\ast$. Then this theorem follows directly from Lemma \ref{lemm1}, Theorem \ref{t2} and Proposition \ref{pro1}.
\end{proof}

If $(p, q, S, T, a_1, k, \eta, D)\in \mathcal{F}(A)$, where $p$, $q$, $T$ and $\eta$ are trivial, then we denote such flag datum of $(A, \circ, [\cdot,\cdot])$ by $(S, a_1, k, D)$. We denote the set of all such flag datums of $(A, \circ, [\cdot,\cdot])$ by $\mathcal{SF}_1(A)$. Note that in this case, $S$ and $D$ are derivations of $(A, \circ)$. By Theorem \ref{t3}, $(S, a_1, k, D)\equiv(S^{'}, a_1^{'}, k^{'},D^{'})$ if and only if there exists a pair $(a_0, \beta)\in A\times K^{\ast}$ such that the following conditions hold:
\begin{eqnarray*}
&&S(a)=a_0\circ a+\beta S^{'}(a),\\
&&a\circ a_0=0,\\
&&a_1=a_0\circ a_0+\beta S^{'}(a_0)+\beta^2a_1^{'}-ka_0,\\
&&k=\beta k^{'},\\
&&D(a)=\beta D^{'}(a)+[a_0,a],
\end{eqnarray*}
for all $a\in A$. In particular, in this case, when $a_1=0$, we denote such flag datum of $(A, \circ, [\cdot,\cdot])$ by $(S, k, D)$. We denote the set of all such flag datums of $(A, \circ, [\cdot,\cdot])$ by $\mathcal{SF}_2(A)$. Define $(S, k, D)\approx(S^{'}, k^{'}, D^{'})$ if  there exists a pair $\beta\in K^{\ast}$ such that the following conditions hold:
\begin{eqnarray*}
&&S(a)=\beta S^{'}(a),\\
&&k=\beta k^{'},\\
&&D(a)=\beta D^{'}(a),
\end{eqnarray*}
for all $a\in A$.

If $(p, q, S, T, a_1, k, \eta, D)\in \mathcal{F}(A)$, where $p$, $T$ are trivial and $k=0$, then we denote such flag datum of $(A, \circ, [\cdot,\cdot])$ by $(q, S, a_1, \eta, D)$. Denote the set of all such flag datums of $(A, \circ, [\cdot,\cdot])$ where $q\neq 0$ by $\mathcal{SF}_3(A)$. By Theorem \ref{t3}, $(q, S, a_1, \eta, D)\equiv(q^{'}, S^{'}, a_1^{'}, \eta^{'}, D^{'})$ if and only if $q=q^{'}$, $\eta=\eta^{'}$ and there exists a pair $(a_0, \beta)\in A\times K^{\ast}$ satisfying the following conditions:
\begin{eqnarray*}
&&S(a)=a_0\circ a+\beta S^{'}(a)-q(a)a_0,\\
&&a\circ a_0=0,\\
&&a_1=a_0\circ a_0+\beta S^{'}(a_0)+\beta^2a_1',\\
&&q(a_0)=0,\\
&&D(a)=\beta D^{'}(a)+[a_0,a]-\eta(a)a_0,
\end{eqnarray*}
for all $a\in A$. Moreover, we define that $(q, S, a_1, \eta, D)\approx (q^{'}, S^{'}, a_1^{'}, \eta^{'}, D^{'})$ if $q=q^{'}$, $\eta=\eta^{'}$ and there exists $\beta\in A$ such that the following conditions hold:
\begin{eqnarray*}
&&S(a)=\beta S^{'}(a),\\
&&a_1=\beta^2a_1^{'},\\
&&D(a)=\beta D^{'}(a),
\end{eqnarray*}
for all $a\in A$.

If $(p, q, S, T, a_1, k, \eta, D)\in\mathcal{F}(A)$, where $p$, $q$, $\eta$ and $D$ are trivial, then we denote such flag datum of $(A, \circ, [\cdot, \cdot])$ by $(S, T, a_1, k)$. Denote the set of all such flag datums of $(A, \circ, [\cdot, \cdot])$ by $\mathcal{SF}_4(A)$. By Theorem \ref{t3}, $(S, T, a_1, k)\equiv(S^{'}, T^{'}, a_1^{'}, k^{'})$ if and only if there exists a pair $(a_0, \beta)\in A\times K^{\ast}$ satisfying the following conditions:
\begin{eqnarray*}
&&S(a)=\beta S^{'}(a)+a_0\circ a,\\
&&T(a)=\beta T^{'}(a)+a\circ a_0,\\
&&a_1=a_0\circ a_0+\beta S^{'}(a_0)+\beta T^{'}a_0+\beta^2a_1^{'}-ka_0,\\
&&k=\beta k^{'},\\
&&[a_0, a]=0,
\end{eqnarray*}
for all $a\in A$. In particular, in this case, when $a_1=0$, we also denote such flag datum of $(A, \circ, [\cdot, \cdot])$ by $(S, T, k)$. Denote the set of all such flag datums of $(A, \circ, [\cdot, \cdot])$ by $\mathcal{SF}_5(A)$. We say that $(S, T, k)\approx(S^{'}, T^{'}, k^{'})$ if there exists $ \beta\in K^{\ast}$ such that the following conditions are satisfied:
\begin{eqnarray*}
&&S(a)=\beta S^{'}(a),\\
&&T(a)=\beta T^{'}(a),\\
&&a_1=\beta^2a_1^{'},\\
&&k=\beta k^{'},
\end{eqnarray*}
for all $a\in A$.

\begin{corollary}
(1) Let $(A, \circ, [\cdot,\cdot])$ be a Gel'fand-Dorfman bialgebra, its sub-adjacent Lie algebra $(g(A), [\cdot,\cdot])$ is perfect, i.e. $[g(A),g(A)]=g(A)$, and all quasicentroids of $(A, \circ)$ are inner. Then
\begin{eqnarray*}
&&Extd(E,A)\cong \mathcal{GH}^2(K,A)\cong(\mathcal{SF}_1(A)/\equiv)\cup(\mathcal{SF}_3(A)/\equiv).
\end{eqnarray*}
In particular, if $\{b\mid a\circ b=0~~ \text{for all $a\in A$}\}=\{0\}$, then
\begin{eqnarray*}
&&Extd(E,A)\cong \mathcal{GH}^2(K,A)\cong(\mathcal{SF}_2(A)/\approx)\cup(\mathcal{SF}_3(A)/\approx).
\end{eqnarray*}
(2) Let $(A, \circ, [\cdot, \cdot])$ be a Gel'fand-Dorfman bialgebra, $(A, [\cdot,\cdot])$ is perfect, i.e. $[A,A]=A$ and all derivations of $(A, [\cdot,\cdot])$ are inner. Then
\begin{eqnarray*}
Extd(E,A)\cong\mathcal{GH}^2(K,A)\cong\mathcal{SF}_4(A)/\equiv.
\end{eqnarray*}
Moreover, if the center of $(A, [\cdot,\cdot])$ is also trivial, then
\begin{eqnarray*}
Extd(E,A)\cong\mathcal{GH}^2(K,A)\cong\mathcal{SF}_5(A)/\approx.
\end{eqnarray*}
\end{corollary}
\begin{proof}
(1) Since $[g(A), g(A)]=g(A)$, following the proof in Corollary 4.6 in \cite{H1}, we can obtain two cases as follows:
\begin{itemize}
\item $p$, $q$ and $T$ are trivial,\\
 \item $p$, $T$ are trivial, $k=0$ and $q$ is non-trivial.
 \end{itemize}
In the first case, by $(GF2)$, we get $\eta(b\circ a)=0$ for all $a$, $b\in A$. Since $[g(A), g(A)]=g(A)$, $A\circ A=A$. Therefore, $\eta=0$. In this case, if $\{b\mid a\circ b=0 ~~~\text{for all $a\in A$}\}=\{0\}$, by $(FN4)$, we get $a_1=0$.  Moreover, by $(FD2)$, we obtain $a\circ a_0=0$ for all $a\in A$. When $\{b\mid a\circ b=0 ~~~\text{for all $a\in A$}\}=\{0\}$,  $a_0=0$. Then these results follow directly from Theorem \ref{t3}.

(2) Since $[A,A]=A$, we get $\eta=0$ by $(TD1)$. Therefore, $D$ is a derivation of $(A, [\cdot, \cdot])$. Since all derivations of $(A, [\cdot,\cdot])$ are inner, we can suppose that $D=0$ by Theorem \ref{t3}. Then by $(GF2)$ and $(GF4)$, we obtain that $p=0$ and $q=0$. In particular, when the center of $(A, [\cdot, \cdot])$ is trivial, we get $[a,a_1]=0$ for all $a\in A$ by $(GF5)$, which means that $a_1=0$. Moreover, by $(FD5)$, $[a_0,a]=0$ for all $a\in A$. Therefore, in this special case, we get $a_0=0$. Then the conclusion follows directly by Theorem \ref{t3}.
\end{proof}

Finally, we present an example to compute $\mathcal{GH}^2(K, A)$.
\begin{example}\label{exx1}
Let $K=\mathbb{C}$ and $(A, \circ, [\cdot,\cdot])$ be a 2-dimensional Gel'fand-Dorfman bialgebra with a basis $\{L, W\}$ and the products given by:
\begin{eqnarray}
&&\label{eqq1}L\circ L=L,~~L\circ W=0,~~~W\circ L=W,~~~W\circ W=0,~~~[L,W]=-bW,
\end{eqnarray}
where $b\in K$. This is the Gel'fand-Dorfman bialgebra corresponding to the Lie conformal algebra $W(1, b)$ studied in \cite{LHZZ}. Note that $W(1, 0)$ is just the Heisenberg-Virasoro conformal algebra introduced in \cite{SY}.

By a rather long but straightforward computation, the following flag datums $(p, q, S, T, a_1, k, \eta, D)$ satisfying $(GF0)$-$(GF6)$ are as follows.

When $b\neq 0$, there are five cases.

\indent{\bf Case A1}:
\begin{eqnarray*}
&&p(L)=1, p(W)=0, q(L)=1, q(W)=0, S=\left(\begin{array}{cc}0 & 0\\0 & 0\end{array}\right),T=\left(\begin{array}{cc}0 & 0\\b_2 & b_1\end{array}\right),\\
&&a_1=b_1^2L-b_1b_2W , k=0,\eta=0, D=\left(\begin{array}{cc}0 & 0\\-bb_2 & -bb_1\end{array}\right),\\
&&for\ all\ b_1, b_2\in K.
\end{eqnarray*}
Thus in this case any 3-dimensional Gel'fand-Dorfman bialgebra that contains $(A, \circ, [\cdot,\cdot])$ as a subalgebra is isomorphic to the following Gel'fand-Dorfman bialgebra denoted by ${A1}_{b_1, b_2}$ with the basis $\{L, W, x\}$ and the products given by (\ref{eqq1})  and
\begin{eqnarray*}
&&x\circ L=x, L\circ x=b_2W+x, x\circ W=0, W\circ x=b_1W, x\circ x= b_1^2L-b_1b_2W,\\
&&[x,L]=-bb_2W, [x,W]=-bb_1W.
\end{eqnarray*}
Two such Gel'fand-Dorfman bialgebras ${A1}_{b_1, b_2}$  and ${A1}_{b_1^{'}, b_2^{'}}$  are equivalent if and only if there exist $c_2\in K$ and $\beta\in K^{\ast}$ such that $b_1=\beta b_1^{'}$, $b_2=\beta b_2^{'}-c_2$. Let $c_2=-b_2$. Then it is easy to see that ${A1}_{b_1, b_2}$ is  equivalent to either ${A1}_{0, 0}$ or ${A1}_{1, 0}$. Therefore, in this case, there are only two equivalence classes, i.e. ${A1}_{0, 0}$ and ${A1}_{1, 0}$.

\indent{\bf Case A2}:
\begin{eqnarray*}
&&p(L)=\frac{1}{2}, p(W)=0,q(L)=1,q(W)=0, S=\left(\begin{array}{cc}0 & 0\\0 & 0\end{array}\right), T=\left(\begin{array}{cc}0 & 0\\b_1 & 0\end{array}\right),\\
&&a_1=b_2W , k=0, \eta(L)=\frac{b}{2},\eta(W)=0, D=\left(\begin{array}{cc}0 & 0\\-bb_1 & 0\end{array}\right), \\
&&for\ all\ b_1, b_2\in K.
\end{eqnarray*}
Thus in this case any 3-dimensional Gel'fand-Dorfman bialgebra that contains $(A, \circ, [\cdot,\cdot])$ as a subalgebra is isomorphic to the following Gel'fand-Dorfman bialgebra denoted by ${A2}_{b_1, b_2}$ with the basis $\{L, W, x\}$ and the products given by (\ref{eqq1})  and
\begin{eqnarray*}
&&x\circ L=x, L\circ x=b_1W+\frac{1}{2}x, x\circ W=0, W\circ x=0, x\circ x= b_2W,\\
&&[x,L]=-bb_1W+\frac{b}{2}x, [x,W]=0.
\end{eqnarray*}
Two such Gel'fand-Dorfman bialgebras ${A2}_{b_1, b_2}$ and ${A2}_{b_1^{'}, b_2^{'}}$ are equivalent if and only if there exist $c_2 \in K$ and $\beta \in K^*$ such that $b_1=\beta b_1^{'}-\frac{c_2}{2}$, and $b_2=\beta^2b_2^{'}$. Let $c_2=-2b_1$. Then it is easy to see that ${A2}_{b_1, b_2}$ is  equivalent to either ${A2}_{0, 0}$ or ${A2}_{0, 1}$. Therefore, in this case, there are only two equivalence classes, i.e. ${A2}_{0, 0}$ and ${A2}_{0, 1}$.

\indent{\bf Case A3}:
\begin{eqnarray*}
&&p(L)=-1, p(W)=0, q(L)=1, q(W)=0, S=\left(\begin{array}{cc}0 & 0\\0 & 0\end{array}\right), T=\left(\begin{array}{cc}2b_1 & 0\\b_2 & b_3\end{array}\right), \\
&&a_1=b_3^2L+b_2b_3W, k=0, \eta(L)=b_4, \eta(W)=0, D=\left(\begin{array}{cc}-b_3b_4 & 0\\bb_2-b_2b_4 & -bb_3\end{array}\right),\\
&&for\ all\ b_1, b_2, b_3, b_4\in K..
\end{eqnarray*}
Thus in this case any 3-dimensional Gel'fand-Dorfman bialgebra that contains $(A, \circ, [\cdot,\cdot])$ as a subalgebra is isomorphic to the following Gel'fand-Dorfman bialgebra denoted by ${A3}_{b_1, b_2, b_3, b_4}$ with the basis $\{L, W, x\}$ and the products given by (\ref{eqq1}) and
\begin{eqnarray*}
&&x\circ L=x, L\circ x=2b_1L+b_2W-x, x\circ W=0, W\circ x=b_3W, \\
&&x\circ x= b_3^2L+b_2b_3W, [x,L]=-b_3b_4L+(bb_2-b_2b_4)W+b_4x, [x,W]=-bb_3W.
\end{eqnarray*}
Two such Gel'fand-Dorfman bialgebras ${A3}_{b_1, b_2, b_3, b_4}$ and ${A3}_{b_1^{'}, b_2^{'}, b_3^{'}, b_4^{'}}$ are equivalent if and only if there exist $c_1$, $c_2\in K$ and $\beta \in K^{\ast}$ such that $b_1=\beta b_1^{'} +c_1$, $b_2=\beta b_2^{'}+c_2$, $b_3=\beta b_3^{'}+c_1$ and $b_4=b_4^{'}$. Let $c_2=b_2$. Then ${A3}_{b_1, b_2, b_3, b_4}$ is equivalent to ${A3}_{b_1, 0, b_3, b_4}$.
Moreover, ${A3}_{b_1, 0, b_3, b_4}$ is equivalent to ${A3}_{b_1, 0, b_3, b_4}$ if and only if there exist $c_1\in K$ and $\beta \in K^{\ast}$ such that $b_1=\beta b_1^{'} +c_1$, $b_3=\beta b_3^{'}+c_1$ and $b_4=b_4^{'}$.

\indent{\bf Case A4}:
\begin{eqnarray*}
&&p=0, q=0, S=\left(\begin{array}{cc}b_1 & 0\\b_2 & 0\end{array}\right), T=\left(\begin{array}{cc}b_1 & 0\\0 & b_1\end{array}\right), \\
&&a_1= b_1^2L+b_1b_2W, k=0, \eta=0, D=\left(\begin{array}{cc}0 & 0\\0 & -bb_1\end{array}\right), \\
&&for\ all\ b_1, b_2\in K.
\end{eqnarray*}
Thus in this case any 3-dimensional Gel'fand-Dorfman bialgebra that contains $(A, \circ, [\cdot, \cdot])$ as a subalgebra is isomorphic to the following Gelfand-Dorfman bialgebra denoted by ${A4}_{b_1, b_2}$ with the basis $\{L, W, x\}$ and the products given by (\ref{eqq1}) and
\begin{eqnarray*}
&&x\circ L=b_1L+b_2W, L\circ x=b_1L, x\circ W=0, W\circ x=b_1W, x\circ x=b_1^2L+b_1b_2W,\\
&&[x,L]=0, [x,W]=-bb_1W.
\end{eqnarray*}
Two such Gel'fand-Dorfman bialgebras ${A4}_{b_1, b_2}$ and ${A4}_{b_1^{'}, b_2^{'}}$ are equivalent if and only if there exist $c_1$, $c_2 \in K$ and $\beta \in K^{\ast}$ such that $b_1=\beta b_1^{'}+c_1$ and $b_2=\beta b_2^{'}+c_2$. Let $c_1=b_1$ and $c_2=b_2$. Then we get that ${A4}_{b_1, b_2}\equiv {A4}_{0, 0}$.

\indent{\bf Case A5}:
\begin{eqnarray*}
&&p=0, q=0, S=\left(\begin{array}{cc}0 & 0\\0 & 0\end{array}\right), T=\left(\begin{array}{cc}0 & 0\\0 & 0\end{array}\right), \\
&&a_1=-kb_1W , \eta=0, D=\left(\begin{array}{cc}0 & 0\\b_1 & 0\end{array}\right),\\
&&for\ all\ b_1, k\in K\ and \ (b_1, k)\neq (0, 0).
\end{eqnarray*}
Thus in this case any 3-dimensional Gel'fand-Dorfman bialgebra that contains $(A, \circ, [\cdot,\cdot])$ as a subalgebra is isomorphic to the following Gel'fand-Dorfman bialgebra denoted by ${A5}_{b_1, k}$ with the basis $\{L, W, x\}$ and the products given by (\ref{eqq1}) and
\begin{eqnarray*}
&&x\circ L=0, L\circ x=0, x\circ W=0, W\circ x=0, x\circ x=-kb_1W+kx,\\
&&[x,L]=b_1W, [x,W]=0.
\end{eqnarray*}
Two such Gel'fand-Dorfman bialgebras ${A5}_{b_1, k}$ and ${A5}_{b_1^{'}, k^{'}}$ are equivalent if and only if there exist $\beta \in K^{\ast}$ such that $b_1=\beta b_1^{'}$ and $k=\beta k^{'}$.

Therefore, by the discussion above and Theorem \ref{t3}, when $b\neq 0$, $\mathcal{GH}^2(\mathbb{C}, A)$ can be described by the disjoint union of ${A1}_{0, 0}$, ${A1}_{1, 0}$, ${A2}_{0, 0}$, ${A2}_{0, 1}$, the equivalence classes of all ${A3}_{b_1, 0, b_3, b_4}$, ${A4}_{0, 0}$ and the equivalence classes of all ${A5}_{b_1, k}$ where $(b_1, k)\neq (0, 0)$.

When $b=0$, there are five cases.

\indent{\bf Case B1}:
\begin{eqnarray*}
&&p(L)=-1, p(W)=0, q(L)=1, q(W)=0, S=\left(\begin{array}{cc}0 & 0\\0 & 0\end{array}\right), T=\left(\begin{array}{cc}2b_1 & 0\\b_2 & b_1\end{array}\right), \\
&&a_1=b_1^2L+b_1b_2W, k=0, \eta(L)=b_3,\eta(W)=b_4, D=\left(\begin{array}{cc}-b_1b_3 & -b_1b_4\\-b_2b_3 & -b_2b_4\end{array}\right), \\
&&for\ all\ b_1, b_2, b_3, b_4\in K.
\end{eqnarray*}
Thus in this case any 3-dimensional Gel'fand-Dorfman bialgebra that contains $(A, \circ, [\cdot,\cdot])$ as a subalgebra is isomorphic to the following Gel'fand-Dorfman bialgebra denoted by ${B1}_{b_1, b_2, b_3, b_4}$ with the basis $\{L, W, x\}$ and the products given by (\ref{eqq1}) and
\begin{eqnarray*}
&&x\circ L=x, L\circ x=2b_1L+b_2W-x, x\circ W=0, W\circ x=b_1W,\\
&& x\circ x= b_1^2L+b_1b_2W, [x,L]=-b_1b_3L-b_2b_3W+b_3x, \\
&&[x,W]=-b_1b_4L-b_2b_4W+b_4x.
\end{eqnarray*}
Two such Gel'fand-Dorfman bialgebras ${B1}_{b_1, b_2, b_3, b_4}$  and ${B1}_{b_1^{'}, b_2^{'}, b_3^{'}, b_4^{'}}$  are equivalent if and only if there exist $c_1$, $c_2\in K$ and $\beta\in K^{\ast}$ such that $b_1=\beta b_1^{'}+c_1$, $b_2=\beta b_2^{'}+c_2$, $b_3=b_3^{'}$ and $b_4=b_4^{'}$.
Let $c_1=b_1$ and $c_2=b_2$. Then ${B1}_{b_1, b_2, b_3, b_4}$ is equivalent to ${B1}_{0, 0, b_3, b_4}$. Moreover, ${B1}_{0, 0, b_3, b_4}\equiv {B1}_{0, 0, b_3^{'}, b_4^{'}}$ if and only if $b_3=b_3^{'}$ and $b_4=b_4^{'}$.

\indent{\bf Case B2}:
\begin{eqnarray*}
&&p=0, q=0, S=\left(\begin{array}{cc}b_1 & 0\\b_2 & 0\end{array}\right), T=\left(\begin{array}{cc}b_1 & 0\\0 & b_1\end{array}\right),\\
&&a_1=b_1^2L+b_1b_2W , k=0, \eta =0, D=\left(\begin{array}{cc}0 & 0\\b_3 & 0\end{array}\right)\\
&&for\ all\ b_1, b_2, b_3\in K.
\end{eqnarray*}
Thus in this case any 3-dimensional Gel'fand-Dorfman bialgebra that contains $(A, \circ, [\cdot,\cdot])$ as a subalgebra is isomorphic to the following Gel'fand-Dorfman bialgebra denoted by ${B2}_{b_1, b_2, b_3}$ with the basis $\{L, W, x\}$ and the products given by (\ref{eqq1}) and
\begin{eqnarray*}
&&x\circ L=b_1L+b_2W, L\circ x=b_1L, x\circ W=0, W\circ x=b_1W, x\circ x= b_1^2L+b_1b_2W,\\
&&[x,L]= b_3W, [x,W]=0.
\end{eqnarray*}
Two such Gel'fand-Dorfman bialgebras ${B2}_{b_1, b_2, b_3}$ and ${B2}_{b_1^{'}, b_2^{'}, b_3^{'}}$ are equivalent if and only if there exist $c_1$, $c_2 \in K$ and $\beta \in K^{\ast}$ such that $b_1=\beta b_1^{'} +c_1$, $b_2=\beta b_2^{'}+c_2$, and $b_3=\beta b_3^{'}$. Let $c_1=b_1$ and $c_2=b_2$. Then it is easy to see that ${B2}_{b_1, b_2, b_3}$ is equivalent to either ${B2}_{0, 0, 0}$ or ${B2}_{0, 0, 1}$.

\indent{\bf Case B3}:
\begin{eqnarray*}
&&p(L)=1,p(W)=0,q(L)=1,q(W)=0,S=\left(\begin{array}{cc}0 & 0\\0 & 0\end{array}\right),T=\left(\begin{array}{cc}0 & 0\\b_1 & b_2\end{array}\right), \\
&&a_1=(b_2^2-kb_2)L+(kb_1-b_1b_2)W,\eta=0, D=\left(\begin{array}{cc}0 & 0\\0 & 0\end{array}\right),\\
&&for\ all\ b_1,b_2,k\in K.
\end{eqnarray*}
Thus in this case any 3-dimensional Gel'fand-Dorfman bialgebra that contains $(A,\circ,[\cdot,\cdot])$ as a subalgebra is isomorphic to the following Gel'fand-Dorfman algebra denoted by $B3_{b_1,b_2,k}$ with the basis $\{L, W, x\}$ and the products given by (\ref{eqq1}) and:
\begin{eqnarray*}
&&x\circ L=x, L\circ x=b_1W+x, x\circ W=0, W\circ x=b_2W,\\
&&x\circ x= (b_2^2-kb_2)L+(kb_1-b_1b_2)W+kx,\\
&&[x,L]=0, [x,W]=0.
\end{eqnarray*}
Two such Gel'fand-Dorfman bialgebras $B3_{b_1,b_2,k}$ and $B3_{b_1^{'}, b_2^{'}, k^{'}}$ are equivalent if and only if there exist $c_1$, $c_2\in K$ and $\beta\in K^{\ast}$ such that $b_1=\beta b_1^{'}-c_2$, $b_2=\beta b_2^{'}+c_1$, $k=\beta k^{'}+2c_1$. Let $c_2=-b_1$. Then it is easy to see that $B3_{b_1,b_2,k}$ is equivalent to $B3_{0,b_2,k}$. Moreover, $B3_{0,b_2,k}\equiv B3_{0,b_2^{'},k^{'}}$ if and only if there exist $c_1\in K$ and $\beta\in K^{\ast}$ such that $b_2=\beta b_2^{'}+c_1$ and $k=\beta k^{'}+2c_1$.

\indent{\bf Case B4}:
\begin{eqnarray*}
&&p(L)=b_1,p(W)=0,q(L)=1,q(W)=0,S=\left(\begin{array}{cc}0 & 0\\0 & 0\end{array}\right), T=\left(\begin{array}{cc}b_2-b_1b_2 & 0\\b_3 & b_2\end{array}\right),\\
&&a_1=-b_1b_2^2L+b_2b_3W,k=b_1b_2+b_2,\eta=0, D=\left(\begin{array}{cc}0 & 0\\0 & 0\end{array}\right),\\
&&for\ all\ b_1, b_2, b_3\in K.
\end{eqnarray*}
Thus in this case any 3-dimensional Gel'fand-Dorfman bialgebra that contains $(A,\circ,[\cdot,\cdot])$ as a subalgebra is isomorphic to the following Gel'fand-Dorfman bialgebra denoted by $B4_{b_1,b_2,b_3}$ with the basis $\{L, W, x\}$ and the products given by (4.1), (4.2) and:
\begin{eqnarray*}
&&x\circ L=x, L\circ x=(b_2-b_1b_2)L+b_3W+b_1x, x\circ W=0, W\circ x=b_2W,\\
&&x\circ x=-b_1b_2^2L+b_2b_3W+( b_1b_2+b_2)x,\\
&&[x,L]=0, [x,W]=0.
\end{eqnarray*}
Two such Gel'fand-Dorfman bialgebras $B4_{b_1,b_2,b_3}$ and $B4_{b_1^{'}, b_2^{'}, b_3^{'}}$ are equivalent if and only if there exist $c_1$, $c_2\in K$ and $\beta\in K^{\ast}$ such that $b_1= b_1^{'}$, $b_2=\beta b_2^{'}+c_1$ and $b3=\beta b_3^{'}-b_1c_2$. Therefore, in this case, there are the following equivalence classes: $B4_{0,0,0}$, $B4_{0,0,1}$, $B4_{b_1,0,0}$ for all $b_1\in K^{\ast}$.

\indent{\bf Case B5}:
\begin{eqnarray*}
&&p(L)=\frac{1}{2}, p(W)=0, q(L)=1, q(W)=0, S=\left(\begin{array}{cc}0 & 0\\0 & 0\end{array}\right), T=\left(\begin{array}{cc}\frac{1}{2}b_1 & 0\\b_2 & b_1\end{array}\right),\\
&&a_1=-\frac{1}{2}b_1^2L+b_3W,k=\frac{3}{2}b_1,\eta=0, D=\left(\begin{array}{cc}0 & 0\\0 & 0\end{array}\right),\\
&&for\ all\ b_1, b_2, b_3\in K.
\end{eqnarray*}
Thus in this case any 3-dimensional Gel'fand-Dorfman bialgebra that contains $(A,\circ,[\cdot,\cdot])$ as a subalgebra is isomorphic to the following Gel'fand-Dorfman bialgebra denoted by $B5_{b_1,b_2,b_3}$ with the basis $\{L,W,x\}$ and the products given by (\ref{eqq1}) and:
\begin{eqnarray*}
&&x\circ L=x,L\circ x=\frac{1}{2}b_1L+b_2W+\frac{1}{2}x,x\circ W=0,W\circ x=b_1W,\\
&&x\circ x=-\frac{1}{2}b_1^2L+b_3W+\frac{3}{2}b_1x,\\
&&[x,L]=0, [x,W]=0.
\end{eqnarray*}
Two such Gel'fand-Dorfman bialgebras $B5_{b_1,b_2,b_3}$ and $B5_{b_1',b_2',b_3'}$ are equivalent if and only if there exist $c_1, c_2\in K$ and $\beta\in K^{\ast}$ such that $b_1=\beta b_1^{'}+c_1$, $b_2=\beta b_2^{'}-\frac{1}{2}c_2$ and $b3= b_1b_2-\beta^2 b_1^{'}b_2^{'}+\beta^2 b_3^{'}$. Let $c_1=b_1$ and $c_2=-2b_2$ and $\beta=1$. Then we can get that $B5_{b_1,b_2,b_3}\equiv B5_{0,0,b_3-b_1b2}$. Moreover, $B5_{0,0,b_3}$ is equivalent to either $B5_{0,0,0}$ or $B5_{0,0,1}$.

Therefore, when $b=0$, $\mathcal{GH}^2(\mathbb{C}, A)$ can be described by the disjoint union of ${B2}_{0, 0, 0}$, ${B2}_{0, 0, 1}$, the equivalence classes of ${B3}_{0, b_2, k}$, $B4_{0,0,0}$, $B4_{0,0,1}$, $B4_{b_1,0,0}$ for all $b_1\in K^{\ast}$,  $B5_{0,0,0}$, $B5_{0,0,1}$ and ${B1}_{0, 0, b_3, b_4}$ for all $b_3$, $b_4\in \mathbb{C}$.
\end{example}

\begin{remark}
By the discussion in the introduction, Example \ref{exx1} also gives a complete classification of all quadratic Lie conformal algebras on $R=\mathbb{C}[\partial](\mathbb{C}L\oplus \mathbb{C}W\oplus \mathbb{C}x)$ containing $W(1,b)=\mathbb{C}[\partial]L\oplus \mathbb{C}[\partial]W$ (in particular, Heisenberg-Virasoro conformal algebra) as a subalgebra up to the isomorphisms which are induced from the isomorphisms of the corresponding Gel'fand-Dorfman bialgebras whose restrictions on $\mathbb{C}L\oplus \mathbb{C}W$ is the identity map.
\end{remark}

\end{document}